\documentclass[a4paper]{article}
\usepackage{amsmath,amscd}
\usepackage{amssymb}
\usepackage{amsthm}
\usepackage{amsbsy,amstext,amsxtra,amsopn}
\makeatletter
  
  \@addtoreset{equation}{section}
\makeatother
\usepackage{latexsym}
\begin{document}
\title{{\bf Meromorphic function fields closed by 
partial derivatives}}
\author{Yukitaka Abe}
\date{ }
\maketitle

\noindent
{\bf Abstract}\\
We characterize meromorphic function fields closed by
partial derivatives in $n$ variables. 
\ 
\footnote{
{\bf keywords:} Algebraic addition theorem, Degenerate
abelian functions, Briot-Bouquet type partial 
differential equations}\ 

\newtheorem{definition}{Definition}

\newtheorem{lemma}{Lemma}

\newtheorem{theorem}{Theorem}

\newtheorem{proposition}{Proposition}

\newtheorem{corollary}{Corollary}

\newtheorem{example}{Example}

\newtheorem{remark}{Remark}

\section{Introduction}
Weierstrass' theorem says that a meromorphic function $f$
on ${\mathbb C}$ admits an algebraic addition theorem if
and only if it is an elliptic function or a degenerate
elliptic function. Such a function $f$ has another
characterization that it is a solution of a Briot-Bouquet
differential equation. If we define $K := {\mathbb C}(f,f')$
by a solution $f$ of a Briot-Bouquet differential equation,
then $K$ is closed by derivative.
\par
Let ${\mathfrak M}({\mathbb C})$ be the field of meromorphic
functions on ${\mathbb C}$, and let $K$ be a subfield of
${\mathfrak M}({\mathbb C})$ which is finitely generated
over ${\mathbb C}$ and has the transcendence degree
${\rm Trans}_{{\mathbb C}}K = 1$. We characterize such $K$
which is closed by derivative in Theorem 1.
\par
Our main result is its generalization to the multidimensional
case. In Theorem 4 we give a characterization of subfields
of ${\mathfrak M}({\mathbb C}^n)$ which are closed by partial
derivatives, where ${\mathfrak M}({\mathbb C}^n)$ is the
field of meromorphic functions on ${\mathbb C}^n$.\par
We use the properties of isogenies to prove the main theorem.
They are collected in Section 4.
We recently showed in \cite{abe3} that any quasi-abelian variety is isogenous to
a product of geometrically simple quasi-abelian subvarieties.
In Section 4 we also prove
that this decomposition of a quasi-abelian variety is unique up to isogeny if it is 
of kind 0.

\section{One dimensional case}
The following differential equation is called a Briot-Bouquet
differential equation;
\begin{equation}
P(f',f) = 0,
\end{equation}
where $P$ is an irreducible polynomial. Every meromorphic
solution $f$ on ${\mathbb C}$ of the above equation is an elliptic
function or a degenerate elliptic function. This result was
first published by Briot and Bouquet \cite{briot-bouquet}.
Conversely, it is well known that an elliptic function or a
degenerate elliptic function satisfies (2.1) for some $P$.
Later, higher order Briot-Bouquet differential equations were
studied (for example \cite{picard}, \cite{hille} and \cite{eremenko}). 
Eremenko, Liao and Ng \cite{eremenko-liao-ng} finally
completed the study for any
order Briot-Bouquet
differential equations. We refer to \cite{eremenko-liao-ng}
for these equations.\par
Another important property of elliptic functions or degenerate
elliptic functions is Weierstrass' theorem. 
We say that $f \in {\mathfrak M}({\mathbb C})$ admits an
algebraic addition theorem if there exists an irreducible
polynomial $P$ such that
\begin{equation}
P(f(\zeta + \eta ), f(\zeta ), f(\eta )) = 0
\end{equation}
for all $\zeta ,\eta \in {\mathbb C}$. It is abbreviated
to ${\rm (AAT^*)}$. Combining the Briot-Bouquet theorem with
Weierstrass' theorem, we obtain the following statement:\\
The
following three conditions are equivalent for a meromorphic
function $f$ on ${\mathbb C}$;\\
(i) $f$ is a solution of some Briot-Bouquet differential
equation,\\
(ii) $f$ admits an algebraic addition theorem,\\
(iii) $f$ is an elliptic function or a degenerate elliptic
function.\par
We reconsider this property for subfields of
${\mathfrak M}({\mathbb C})$.
For a subfield $K$ of ${\mathfrak M}({\mathbb C})$ we
consider the following condition (T) concerning the
transcendence degree.\par
\medskip 
\noindent
(T) $K$ is finitely generated over ${\mathbb C}$ and
${\rm Trans}_{{\mathbb C}}K = 1$.\par
\medskip 
\noindent
If $K$ satisfies the condition (T), then we may write
$K = {\mathbb C}(f_0,f_1)$ by some functions $f_0,f_1 \in K$.

\begin{definition}
Let $K = {\mathbb C}(f_0,f_1)$ be a subfield of ${\mathfrak M}({\mathbb C})$
satisfying the condition {\rm (T)}. We say that $K$ admits an
algebraic addition theorem (it is abbreviated to {\rm (AAT)}) if
for any $j=0,1$ there exists a rational function $R_j$ such
that
\begin{equation}
f_j(\zeta + \eta ) = R_j(f_0(\zeta ), f_1(\zeta ),f_0(\eta ),
f_1(\eta ))
\end{equation}
for all $\zeta ,\eta \in {\mathbb C}$.
\end{definition}

It is easily checked by an elementary algebraic argument that
the above definition does not depend on the choice of
generators $f_0, f_1$ of $K$. 

\begin{lemma}
Let $K$ be a subfield of ${\mathfrak M}({\mathbb C})$ satisfying
the condition {\rm (T)}. If $K$ admits {\rm (AAT)}, then any
$f \in K$ admits ${\rm (AAT^*)}$.\par
Conversely,
if a non-constant $f \in K$ admits ${\rm (AAT^*)}$,
then there exists an algebraic extension $\widetilde{K}$ of
$K$ which admits {\rm (AAT)}.
\end{lemma}

\begin{proof}
Let $K = {\mathbb C}(f_0, f_1)$. We assume that $K$ admits
{\rm (AAT)}. We first consider the case that both $f_0$ and
$f_1$ are non-constant. Take any $f \in K$ which is a non-constant
function. Since $f$ and $f_1$ are algebraically dependent, there
exists an irreducible polynomial $P$ such that
\begin{equation}
P(f(\zeta), f_1(\zeta)) = 0.
\end{equation}
Then we have
\begin{equation}
P(f(\zeta + \eta), f_1(\zeta + \eta)) = 0
\end{equation}
for all $\zeta , \eta \in {\mathbb C}$. We also have an irreducible
polynomial $P_0$ such that
\begin{equation}
P_0(f_0(\zeta), f_1(\zeta)) = 0.
\end{equation}
By the assumption we can take a rational function $R$ such that
\begin{equation}
f_1(\zeta + \eta) = R(f_0(\zeta), f_1(\zeta), f_0(\eta), f_1(\eta)).
\end{equation}
Eliminating $f_1(\zeta + \eta), f_1(\zeta)$ and $f_1(\eta)$ by
(2.4), (2.5), (2.6) and (2.7), we obtain an algebraic relation
$$
Q(f(\zeta + \eta), f(\zeta), f(\eta)) = 0,$$
where $Q$ is an irreducible polynomial.\par
When $f_0$ or $f_1$ is constant, we obtain the conclusion by the
same argument.\par
Next we assume that a non-constant $f \in K$ admits 
${\rm (AAT^*)}$. Let $g \in K$ be another non-constant function.
Then there exists an irreducible polynomial $\widetilde{P}$ such that
\begin{equation}
\widetilde{P}(f(\zeta), g(\zeta)) = 0.
\end{equation}
Hence we have
\begin{equation}
\widetilde{P}(f(\zeta + \eta), g(\zeta + \eta)) = 0
\end{equation}
for all $\zeta , \eta \in {\mathbb C}$. Since $f$ admits ${\rm (AAT^*)}$,
there exists an irreducible polynomial $\widetilde{Q}$ such that
\begin{equation}
\widetilde{Q}(f(\zeta + \eta), f(\zeta), f(\eta)) = 0
\end{equation}
for all $\zeta , \eta \in {\mathbb C}$. Eliminating $f(\zeta + \eta)$ by
(2.9) and (2.10), we obtain
$$\widetilde{Q}_0(g(\zeta + \eta), f(\zeta), f(\eta)) = 0,$$
where $\widetilde{Q}_0$ is an irreducible polynomial. Using (2.8), we
finally obtain
$$S(g(\zeta + \eta ), g(\zeta), g(\eta)) = 0$$
for some irreducible polynomial $S$. Then $g$ also admits
${\rm (AAT^*)}$.\par
We set $K_0 := {\mathbb C}(f)$. Then $K_0 \subset K$ and
${\rm Trans}_{{\mathbb C}}K_0 = 1$. By the Briot-Bouquet theorem, $K_0$ is
contained in ${\mathbb C}(\zeta)$, ${\mathbb C}(e^{\alpha \zeta})$ with
$\alpha \in {\mathbb C}^{*}$ or an elliptic function field $E$.
Assume that $K_0 \subset {\mathbb C}(\zeta)$. If there exists $g \in K$
with $g \notin {\mathbb C}(\zeta)$, then $g$ belongs to 
${\mathbb C}(e^{\alpha \zeta})$ or $E$ by the Briot-Bouquet theorem.
Then $f$ and $g$ are algebraically independent. This contradicts our
assumption. Therefore we have $K \subset {\mathbb C}(\zeta)$.
By the same argument for other cases, we conclude that $K$ is
a subfield of one of these function fields. Thus $K$ has the algebraic
extension $\widetilde{K}$ which admits {\rm (AAT)}.
\end{proof}

\begin{definition}
Let $K$ be a subfield of ${\mathfrak M}({\mathbb C})$. $K$ is
closed by derivative if $f' \in K$ for any $f \in K$.
\end{definition}

Suppose that $f \in {\mathfrak M}({\mathbb C})$ is a solution
of a Briot-Bouquet differential equation (2.1). If we set
$K := {\mathbb C}(f,f')$, then $K$ satisfies the condition (T)
and is closed by derivative.
\par

We obtain the following theorem. We will give its proof
after stating some results in the case of $n$ variables.

\begin{theorem}
Let $K$ be a subfield of ${\mathfrak M}({\mathbb C})$ satisfying
the condition {\rm (T)}. Then the following conditions are
equivalent.\\
(1) There exists a non-constant $f \in K$ with $f' \in K$.\\
(2) $K$ is closed by derivative.\\
(3) There exists a ${\mathbb C}$-linear isomorphism
$\Phi : {\mathbb C} \longrightarrow {\mathbb C}$ such that
$\Phi ^{*}K := \{ f \circ \Phi ; f \in K \}$ is ${\mathbb C}(\zeta )$,
${\mathbb C}(e^{\zeta })$ or a subfield of an elliptic function field
which is closed by derivative.
\end{theorem}

Theorem 1 means that ${\mathbb C}(\zeta )$,
${\mathbb C}(e^{\zeta })$ or an elliptic function field is essentially maximal
in the subfields which are closed by derivative and satisfy the
condition {\rm (T)}. \par
We generalize Theorem 1 to the $n$-dimensional case (see Theorem 4). 
It is our
purpose in this paper.

\section{Algebraic addition theorem}
Weierstrass had frequently stated the following in his lectures
in Berlin (see \cite{painleve}):\\
Every system of $n$ (independent) functions in $n$ variables
which admits an addition theorem is an algebraic combination of
$n$ abelian (or degenerate abelian) functions with the same periods.
\par
However, he had never published his proof of the above statement.
We did not have clear concept of degenerate abelian functions
untill quite recently. The precise meaning of Weierstrass'
statement became clear in \cite{abe1} and \cite{abe2}. We also
obtained the explicit representation of degenerate abelian
functions in \cite{abe2}. We recall some results which are
needed in our arguments.\par
Let $K$ be a subfield of 
${\mathfrak M}({\mathbb C}^n)$. We consider
the following condition (T) as in the case $n=1$.\par
\medskip 
\noindent
(T) $K$ is finitely generated over ${\mathbb C}$ and
${\rm Trans}_{{\mathbb C}}K = n$.\par
\medskip
\noindent
If $K$ satisfies the condition (T), then we can take
$f_0, f_1, \dots , f_n \in K$ with $K = {\mathbb C}
(f_0, f_1, \dots , f_n)$.

\begin{definition}
Let $K = {\mathbb C}(f_0, f_1, \dots , f_n)$ be a subfield
of ${\mathfrak M}({\mathbb C}^n)$ satisfying the condition
{\rm (T)}. We say that $K$ admits an algebraic addition
theorem (it is also abbreviated to {\rm (AAT)}) if for any
$j = 0, 1, \dots , n$ there exists a rational function $R_j$
such that
\begin{equation}
f_j(z + w) = R_j(f_0(z), f_1(z), \dots , f_n(z), f_0(w),
f_1(w), \dots ,f_n(w))
\end{equation}
for all $z,w \in {\mathbb C}^n$.
\end{definition}

The above definition does not depend on the choice of
generators $f_0, f_1, \dots , f_n$ of $K$.\par
A toroidal group is a connected complex Lie group without
non-constant holomorphic function. It is well known that
toroidal groups are commutative. Then every toroidal group
is written as ${\mathbb C}^n/\Gamma $, where $\Gamma $ is a
discrete subgroup of ${\mathbb C}^n$ with ${\rm rank}\; \Gamma
= n+m$, $1 \leqq m \leqq n$.
We denote by ${\mathbb R}_{\Gamma }^{n+m}$ the real linear
subspace generated by $\Gamma $. The maximal complex linear
subspace contained in ${\mathbb R}_{\Gamma }^{n+m}$ has the
complex dimension $m$. It is written as ${\mathbb C}_{\Gamma }^m
= {\mathbb R}_{\Gamma }^{n+m} \cap \sqrt{-1}{\mathbb R}_{\Gamma }^{n+m}$.
When ${\rm rank}\; \Gamma = n+m$, a toroidal group
${\mathbb C}^n/\Gamma $ has the structure of principal
$({\mathbb C}^{*})^{n-m}$-bundle $\rho : {\mathbb C}^n/\Gamma 
\longrightarrow {\mathbb T}$ over an $m$-dimensional complex
torus ${\mathbb T}$. Replacing fibers $({\mathbb C}^{*})^{n-m}$
with $({\mathbb P}^1)^{n-m}$, we obtain the associated
$({\mathbb P}^1)^{n-m}$-bundle $\overline{\rho } : 
\overline{{\mathbb C}^n/\Gamma } \longrightarrow {\mathbb T}$.
We call $\overline{{\mathbb C}^n/\Gamma }$ the standard
compactification of a toroidal group ${\mathbb C}^n/\Gamma $.
\par
A toroidal group ${\mathbb C}^n/\Gamma $ is called a quasi-abelian
variety if there exists a hermitian form ${\mathcal H}$ on
${\mathbb C}^n$ such that
\par
\medskip
\noindent
(i) ${\mathcal H}$ is positive definite on ${\mathbb C}_{\Gamma }^m$,\\
(ii) the imaginary part ${\mathcal A} := {\rm Im}{\mathcal H}$ of ${\mathcal H}$ is
${\mathbb Z}$-valued on $\Gamma \times \Gamma $.\par
\medskip
\noindent
The above hermitian form ${\mathcal H}$ is said to be an ample
Riemann form for ${\mathbb C}^n/\Gamma $. We denote by
${\mathcal A}_{\Gamma }$ the restriction of ${\mathcal A}$ on
${\mathbb R}_{\Gamma }^{n+m} \times {\mathbb R}_{\Gamma }^{n+m}$.
If ${\mathcal H}$ is an ample Riemann form for a quasi-abelian
variety ${\mathbb C}^n/\Gamma $, then ${\rm rank}{\mathcal A}_
{\Gamma } = 2(m + k)$ with $0 \leqq 2k \leqq n-m$. 
In this case we
say that the ample Riemann form ${\mathcal H}$ is of kind $k$.
The kind of a quasi-abelian variety ${\mathbb C}^n/\Gamma $
is defined
by the smallest kind of ample Riemann forms for ${\mathbb C}^n/
\Gamma $ (\cite{abe-umeno}). A quasi-abelian variety 
${\mathbb C}^n/\Gamma $ with ${\rm rank}\; \Gamma = n+m$ is
of kind 0 if and only if it is a principal $({\mathbb C}^{*})^{n-m}$-bundle
$\rho : {\mathbb C}^n/\Gamma \longrightarrow {\mathbb A}$ over
an $m$-dimensional abelian variety ${\mathbb A}$. Its standard
compactification $\overline{{\mathbb C}^n/\Gamma }$ is the
associated $({\mathbb P}^1)^{n-m}$-bundle $\overline{\rho } :
\overline{{\mathbb C}^n/\Gamma } \longrightarrow {\mathbb A}$.\par
By the Remmert-Morimoto theorem any connected commutative complex
Lie group $X$ of dimension $n$ is represented as
$$ X = {\mathbb C}^n/\Gamma = {\mathbb C}^p \times
({\mathbb C}^{*})^q \times ({\mathbb C}^r / \Gamma _0),$$
where ${\mathbb C}^r/\Gamma _0$ is a toroidal group and
$p+q+r = n$. By the standard compactification
$\overline{{\mathbb C}^r/\Gamma _0}$ of ${\mathbb C}^r/\Gamma _0$
we obtain a compactification
$$\overline{X} = ({\mathbb P}^1)^{p} \times ({\mathbb P}^1)^{q}
\times \overline{{\mathbb C}^r/\Gamma _0}$$
of $X$. It is also called the standard compactification of
a connected commutative complex Lie group $X$.\par
Let $f \in {\mathfrak M}({\mathbb C}^n)$. We define the period
group $\Gamma _f$ of $f$ by
$$\Gamma _f := \{ \gamma \in {\mathbb C}^n;\, f(z + \gamma ) =
f(z)\quad \text{for all $z \in {\mathbb C}^n$}\}.$$

\begin{definition}
A meromorphic function $f$ on ${\mathbb C}^n$ is said to be
non-degenerate if its period group $\Gamma _f$ is discrete.
\end{definition}

For a subfield $K$ of ${\mathfrak M}({\mathbb C}^n)$ we denote by
$\Gamma _{K} := \bigcap _{f \in K}\Gamma _f$ the period group of $K$.
A subfield $K$ is said to be non-degenerate if it has a non-degenerate
meromorphic function $f$. Since $\Gamma _K \subset \Gamma _f$
for any $f \in K$, a subfield $K$ is non-degenerate if and only if 
$\Gamma _K$ is discrete.\par
Let $X = {\mathbb C}^n/\Gamma $ be a connected commutative complex
Lie group as above. We denote by ${\mathfrak M}(X)$ the field of
meromorphic functions on $X$. Let $\sigma : {\mathbb C}^n
\longrightarrow X$ be the canonical projection.
Then, for any subfield $K$ of ${\mathfrak M}({\mathbb C}^n)$ with
$\Gamma \subset \Gamma _K$ there exists a subfield $\kappa $ of
${\mathfrak M}(X)$ such that $K = \sigma ^{*}\kappa $.
Let ${\mathfrak M}(\overline{X})$ be the field of meromorphic
functions on the standard compactification $\overline{X}$ of $X$.
We denote by ${\mathfrak M}(\overline{X})|_X$ the restriction
of ${\mathfrak M}(\overline{X})$ onto $X$.

\begin{definition}
A subfield $K$ of ${\mathfrak M}({\mathbb C}^n)$ is said to be
a {\rm W}-type subfield if $K = \sigma ^{*}({\mathfrak M}(\overline{X})|_X)$,
where $X = {\mathbb C}^p \times ({\mathbb C}^{*})^q \times
{\mathcal Q}$ with an $r$-dimensional quasi-abelian variety
${\mathcal Q}$ of kind 0, $n = p + q + r$ and $\sigma : {\mathbb C}^n
\longrightarrow X$ is the projection.
\end{definition}

The following theorem is proved in \cite{abe1} and \cite{abe2}.

\begin{theorem}
Let $K$ be a non-degenerate subfield of ${\mathfrak M}({\mathbb C}^n)$
satisfying the condition {\rm (T)}. If $K$ admits {\rm (AAT)},
then there exists a ${\mathbb C}$-linear isomorphism
$\Phi : {\mathbb C}^n \longrightarrow {\mathbb C}^n$ such that
$\Phi ^{*}K$ is a subfield of a {\rm W}-type subfield.
\end{theorem}

The above theorem means that a {\rm W}-type subfield is
essentially maximal in the subfields which admit {\rm (AAT)}
and the condition {\rm (T)}.

\section{Isogeny}
Let $X = {\mathbb C}^n/\Gamma $ and $X' = {\mathbb C}^{n'}/\Gamma '$
be connected commutative complex Lie groups with
${\rm rank}\; \Gamma = r$ and ${\rm rank}\; \Gamma '= r'$ respectively.
We denote by ${\mathbb R}_{\Gamma }^{r}$ and 
${\mathbb R}_{\Gamma '}^{r'}$ the real linear subspaces spanned by
$\Gamma $ and $\Gamma '$ respectively. The complex dimension of the
complex linear subspace spanned by $\Gamma $ is said to be the complex
rank of $\Gamma $. We denote it by ${\rm rank}_{{\mathbb C}}\Gamma $.
If $\left({\mathbb R}_{\Gamma }^{r}\right)^{{\mathbb C}}$ is the
complexification of ${\mathbb R}_{\Gamma }^{r}$, then
${\rm rank}_{{\mathbb C}}\Gamma = \dim \left({\mathbb R}_{\Gamma }^{r}\right)^{{\mathbb C}}.$
\par
We say that a mapping $\varphi : X \longrightarrow X'$ is a homomorphism
between $X$ and $X'$ if it is holomorphic and a group homomorphism.
For any homomorphism $\varphi : X \longrightarrow X'$ there
uniquely exists a ${\mathbb C}$-linear mapping $\Phi : {\mathbb C}^n
\longrightarrow {\mathbb C}^{n'}$ with $\Phi (\Gamma ) \subset \Gamma '$
such that $\sigma ' \circ \Phi = \varphi \circ \sigma ,$ where
$\sigma : {\mathbb C}^n \longrightarrow X$ and $\sigma ' :
{\mathbb C}^{n'} \longrightarrow X'$ are the projections (for example, see
Proposition 2.1.4 in \cite{abe4}). In this case $\Phi $ is called the
linear extension of $\varphi .$\par
Conversely, a ${\mathbb C}$-linear mapping $\Phi : {\mathbb C}^n
\longrightarrow {\mathbb C}^{n'}$ with $\Phi (\Gamma ) \subset \Gamma '$
defines a homomorphism $\varphi : X \longrightarrow X'$.\par
We denote by $M(k,\ell ; R)$ the set of all $(k,\ell )$-matrices
with coefficients in a ring $R$. When $k = \ell $, we write
$M(k; R) = M(k,k;R)$.\par
Let $\varphi : X = {\mathbb C}^n/\Gamma \longrightarrow
X' = {\mathbb C}^{n'}/\Gamma '$ be a homomorphism.
The linear extension $\Phi $ of $\varphi $ has the representative
matrix $M_{\Phi } \in M(n,n';{\mathbb C})$ with respect to
natural basis of ${\mathbb C}^n$ and ${\mathbb C}^{n'}$.
Let $\gamma _1, \dots , \gamma _r$ and $\gamma '_1, \dots ,\gamma '_{r'}$ be
generators of $\Gamma $ and $\Gamma '$ respectively. If we represent
$\gamma _i$ and $\gamma '_j$ as column vectors, then
$P:=(\gamma _1, \dots ,\gamma _r) \in M(n,r;{\mathbb C})$ and
$P' := (\gamma '_1, \dots , \gamma '_{r'}) \in M(n',r';{\mathbb C}).$
These matrices $P$ and $P'$ are called period matrices of $X$
and $X'$ (or $\Gamma $ and $\Gamma '$) respectively.
The condition $\Phi (\Gamma ) \subset \Gamma '$ is satisfied if and
only if there exists $A_{\Phi } \in M(r',r;{\mathbb Z})$ such that
\begin{equation}
M_{\Phi }P = P' A_{\Phi }.
\end{equation}

\begin{definition}
A homomorphism $\varphi : X \longrightarrow X'$ is said to be an
isogeny if it is surjective and the kernel ${\rm Ker}(\varphi )$ of
$\varphi $ is a finite group. The degree $\deg (\varphi )$ of
$\varphi $ is the number of elements of ${\rm Ker}(\varphi )$.
We say that $X$ and $X'$ are isogenous if there exists an isogeny
$\varphi : X \longrightarrow X'$.
\end{definition}

If $\varphi : X = {\mathbb C}^n/\Gamma \longrightarrow X' =
{\mathbb C}^{n'}/\Gamma '$ is an isogeny, then $n=n'$ and
the linear extension $\Phi $ of $\varphi $ is a ${\mathbb C}$-linear
isomorphism. The kernel ${\rm Ker}(\varphi )$ is a finite group
if and only if $\Gamma '/\Phi (\Gamma )$ is a finite group.
It is equivalent to
${\rm rank}\; \Gamma ' = {\rm rank}\; \Phi (\Gamma ) = {\rm rank}\; \Gamma $.
In this case, $M_{\Phi }$ and $A_{\Phi }$ satisfying (4.1) are
$M_{\Phi } \in GL(n,{\mathbb C})$ and $A_{\Phi } \in M(r;{\mathbb Z})$
with $\det A_{\Phi } \not= 0$. We may write
$$X = {\mathbb C}^n/\Gamma = {\mathbb C}^p \times ({\mathbb C}^{*})^q
\times {\mathcal T},\quad
X' = {\mathbb C}^n/\Gamma '= {\mathbb C}^{p'} \times ({\mathbb C}^{*})^{q'}
\times {\mathcal T}',$$
where ${\mathcal T} = {\mathbb C}^r/\Gamma _0$ and
${\mathcal T}' = {\mathbb C}^{r'}/\Gamma '_0$ are toroidal groups,
$n = p + q + r = p' + q' + r'$ and 
${\rm rank}\; \Gamma = q + {\rm rank}\; \Gamma _0 = q' +
{\rm rank}\; \Gamma '_0 = {\rm rank}\; \Gamma '$.
Let ${\rm rank}\; \Gamma _0 = r + s$ and ${\rm rank}\; \Gamma '_0 = r' + s'$.
Since $\Phi ({\mathbb R}_{\Gamma }^{q+r+s}) = {\mathbb R}_{\Gamma '}^{q'+r'+s'}$
and $\Phi : {\mathbb C}^n \longrightarrow {\mathbb C}^n$ is a ${\mathbb C}$-linear isomorphism,
we have $\Phi \left( \left( {\mathbb R}_{\Gamma }^{q+r+s}\right)^{{\mathbb C}}\right)
= \left( {\mathbb R}_{\Gamma '}^{q'+r'+s'}\right)^{{\mathbb C}}$.
It is obvious that ${\rm rank}_{{\mathbb C}}\Gamma = q + r$ and
${\rm rank}_{{\mathbb C}}\Gamma ' = q' + r'$. Then we obtain
$q + r = q' + r'$. Hence we have $p = p'$. If $\Phi ({\mathbb C}^p) \nsubseteq {\mathbb C}^p$,
then ${\rm Ker}(\varphi )$ is not a finite group. Therefore $\Phi ({\mathbb C}^p) \subset
{\mathbb C}^p$, hence $\Phi ({\mathbb C}^p) = {\mathbb C}^p$.
Thus we obtain the representation 
$\Phi = (\Phi _1, \Phi _2) : {\mathbb C}^p \times {\mathbb C}^{n-p}
\longrightarrow {\mathbb C}^p \times {\mathbb C}^{n-p}$, where
$\Phi _1 : {\mathbb C}^p \longrightarrow {\mathbb C}^p$ and
$\Phi _2 : {\mathbb C}^{n-p} \longrightarrow {\mathbb C}^{n-p}$
are ${\mathbb C}$-linear isomorphisms.

\begin{lemma}
Let $X = {\mathbb C}^n/\Gamma $ and $X' = {\mathbb C}^n/\Gamma '$ be
connected commutative complex Lie groups with
${\rm rank}_{{\mathbb C}}\Gamma = {\rm rank}_{{\mathbb C}}\Gamma ' = n$.
If $\varphi : X \longrightarrow X'$ is an isogeny, then the following
statements hold.\\
(1) $\deg (\varphi ) = |\det A_{\Phi }|$, where $\Phi $ is the linear
extension of $\varphi $.\\
(2) There exists an isogeny $\psi : X' \longrightarrow X$ such that
$\psi \circ \varphi = \alpha \cdot id_X$ and 
$\varphi \circ \psi = \alpha \cdot id_{X'}$, where $\alpha = \deg (\varphi )$.
\end{lemma}

\begin{proof}
Let ${\rm rank}\; \Gamma = {\rm rank}\; \Gamma '= r$.
If $P$ and $P'$ are period matrices of $X$ and $X'$ respectively, then there
exist $M_{\Phi } \in GL(n,{\mathbb C})$ and $A_{\Phi } \in M(r;{\mathbb Z})$
with $\det A_{\Phi } \not= 0$ such that
$M_{\Phi } P = P'A_{\Phi }$. Since we can take generators of $\Gamma $ such
that $A_{\Phi }$ is a diagonal matrix, the statement (1) is trivial.\par
Let $\alpha := |\det A_{\Phi }|$. Then $\alpha = \deg (\varphi )$ by (1).
We have $P(\alpha A_{\Phi }^{-1}) = \alpha M_{\Phi }^{-1}P'$. The matrix
$ \alpha M_{\Phi }^{-1}$ defines a ${\mathbb C}$-linear isomorphism
$\Psi : {\mathbb C}^n \longrightarrow {\mathbb C}^n$ with
$\Psi (\Gamma ') \subset \Gamma $. It is easily checked that $\Psi $
gives an isogeny $\psi : X' \longrightarrow X$ possessing the properties
in the statement (2).
\end{proof}

\begin{proposition}
Let $X = {\mathbb C}^p \times ({\mathbb C}^{*})^q \times {\mathcal T}$
and $X' = {\mathbb C}^{p'} \times ({\mathbb C}^{*})^{q'} \times {\mathcal T}'$ 
be connected commutative complex Lie groups, where ${\mathcal T}$ and
${\mathcal T}'$ are an $r$-dimensional toroidal group and an $r'$-dimensional
toroidal group respectively. If $\varphi : X \longrightarrow X'$ is an isogeny,
then $p = p'$, $q = q'$, $r = r'$ and
$\varphi = (\varphi _1, \varphi _2, \varphi _3) : {\mathbb C}^p \times
({\mathbb C}^{*})^q \times {\mathcal T} \longrightarrow {\mathbb C}^p
\times ({\mathbb C}^{*})^q \times {\mathcal T}'$, where
$\varphi _1 : {\mathbb C}^p \longrightarrow {\mathbb C}^p$ is a ${\mathbb C}$-linear
isomorphism, $\varphi _2 : ({\mathbb C}^{*})^q \longrightarrow ({\mathbb C}^{*})^q$
and $\varphi _3 : {\mathcal T} \longrightarrow {\mathcal T}'$ are isogenies.
Furthermore, $\deg (\varphi ) = \deg (\varphi _2) \deg (\varphi _3)$.
\end{proposition}

\begin{proof}
We have already shown that $p = p'$, $q + r = q' + r'$ and $\varphi |_{{\mathbb C}^p}
: {\mathbb C}^p \longrightarrow {\mathbb C}^p$ is a ${\mathbb C}$-linear
isomorphism. Then, it suffices to consider an isogeny $\psi := \varphi |_{
({\mathbb C}^{*})^q \times {\mathcal T}} : ({\mathbb C}^{*})^q \times {\mathcal T}
\longrightarrow ({\mathbb C}^{*})^{q'} \times {\mathcal T}'$.\par
We set $Y := ({\mathbb C}^{*})^q \times {\mathcal T}$ and
$Y' := ({\mathbb C}^{*})^{q'} \times {\mathcal T}'$.
Since $\psi ({\mathcal T})$ is a toroidal group, we have 
$\psi ({\mathcal T}) \subset {\mathcal T}'$. By Lemma 2 there exists an
isogeny $\tau : Y' \longrightarrow Y$ such that $\psi \circ \tau = \alpha \cdot
id _{Y'}$ and $\tau \circ \psi = \alpha \cdot id _Y$, where $\alpha = \deg (\psi )$.
We also have $\tau ({\mathcal T}') \subset {\mathcal T}$. Then
$\alpha {\mathcal T}' \subset \psi ({\mathcal T}) \subset {\mathcal T}'$.
Therefore we obtain $\psi ({\mathcal T}) = {\mathcal T}'$ for $\alpha {\mathcal T}' = {\mathcal T}'$.
Thus we have $q = q'$ and $r = r'$. Both $\psi |_{({\mathbb C}^{*})^q} : ({\mathbb C}^{*})^q
\longrightarrow ({\mathbb C}^{*})^q$ and $\psi |_{{\mathcal T}} : {\mathcal T}
\longrightarrow {\mathcal T}'$ are isogenies. The statement for the degree of
$\varphi $ is trivial.
\end{proof}

\begin{remark}
Assume that toroidal groups ${\mathcal T}$ and ${\mathcal T}'$ are isogenous.
If either of them is a quasi-abelian variety of kind 0, then so is another one.
\end{remark}

\begin{proposition}
Let $X = {\mathbb C}^n/\Gamma $ and $X' = {\mathbb C}^n/\Gamma '$ be quasi-abelian
varieties of kind 0 with ${\rm rank}\; \Gamma = {\rm rank}\; \Gamma ' = n + m$.
They are principal $({\mathbb C}^{*})^{n-m}$-bundles
$\rho : X \longrightarrow {\mathbb A}$ and $\rho ' : X' \longrightarrow {\mathbb A}'$
over $m$-dimensional abelian varieties ${\mathbb A}$ and ${\mathbb A}'$ respectively.
If $\varphi : X \longrightarrow X'$ is an isogeny, then there exists an isogeny
$\varphi _{{\mathbb A}} : {\mathbb A} \longrightarrow {\mathbb A}'$ such that
$(\varphi , \varphi _{{\mathbb A}}) : (X,\rho , {\mathbb A}) \longrightarrow
(X',\rho ',{\mathbb A}')$ is a bundle morphism, where $\varphi _{{\mathbb A}}$ is 
given by the linear extension $\Phi $ of $\varphi $.
\end{proposition}

\begin{proof}
Let $W$ and $W'$ be the real linear subspaces with
${\mathbb R}_{\Gamma }^{n+m} = {\mathbb C}_{\Gamma }^m \oplus W$
and ${\mathbb R}_{\Gamma '}^{n+m} = {\mathbb C}_{\Gamma ' }^m \oplus W'$
respectively. The linear extension $\Phi $ of $\varphi $ has the properties
$\Phi ({\mathbb C}_{\Gamma }^m) = {\mathbb C}_{\Gamma'}^m$ and
$\Phi (W \oplus \sqrt{-1}W) = W' \oplus \sqrt{-1}W'$ as shown above.
Let $\sigma : {\mathbb C}^n \longrightarrow {\mathbb C}_{\Gamma }^m$ and
$\sigma' : {\mathbb C}^n \longrightarrow {\mathbb C}_{\Gamma'}^m$ be
projections. Then we have ${\mathbb A} = {\mathbb C}_{\Gamma}^m/\sigma(\Gamma)$
and ${\mathbb A}' = {\mathbb C}_{\Gamma'}^m/\sigma'(\Gamma')$.
Hence $\Phi |_{{\mathbb C}_{\Gamma}^m} : {\mathbb C}_{\Gamma}^m
\longrightarrow {\mathbb C}_{\Gamma'}^m$ gives an isogeny
$\varphi _{{\mathbb A}} : {\mathbb A} \longrightarrow {\mathbb A}'$ such that
$(\varphi ,\varphi _{{\mathbb A}}) : (X, \rho ,{\mathbb A}) \longrightarrow
(X', \rho', {\mathbb A}')$ is a bundle morphism.
\end{proof}

We recently showed in \cite{abe3} that any quasi-abelian variety is isogenous
to a product of geometrically simple quasi-abelian subvarieties. Here we say
that a toroidal group is geometrically simple if it does not contain a closed
toroidal subgroup apart from itself and zero (Definition 1 in \cite{abe3}).
We note that a simple complex torus is a geometrically simple toroidal group.
Unfortunately, we do not have a proof for the uniqueness of such a product
up to isogeny yet. We give its proof here. First we need the following lemma.

\begin{lemma}
Let $X = {\mathbb C}^n/\Gamma $ be a toroidal group with
${\rm rank}\; \Gamma = n + m$, which is a principal $({\mathbb C}^{*})^{n-m}$-bundle
$\rho : X \longrightarrow {\mathbb T}$ over an $m$-dimensional complex torus
${\mathbb T}$. Then $X$ is geometrically simple if and only if ${\mathbb T}$
is simple.
\end{lemma}

\begin{proof}
We assume that $X$ is geometrically simple. If ${\mathbb T}$ is not simple,
then there exists a non-zero proper subtorus ${\mathbb T}_0$ of ${\mathbb T}$.
Let $k := \dim {\mathbb T}_0$. Then $\rho ^{-1}({\mathbb T}_0)$ is a closed
complex Lie subgroup of $X$ with $\dim \rho ^{-1}({\mathbb T}_0) = k + n - m$.
By the Remmert-Morimoto theorem we have
$$\rho ^{-1}({\mathbb T}_0) \cong {\mathbb C}^p \times ({\mathbb C}^{*})^q
\times X_0,$$
where $X_0$ is an $r$-dimensional toroidal group and $p + q + r = k + n - m$.
Since the maximal dimension of a closed Stein subgroup contained in $X$
is $n - m$ (Proposition 1.1.17 in \cite{abe-kopfermann} or Proposition 2.6.1
in \cite{abe4}), we have $p + q \leqq n - m$. Then we obtain
$r \geqq k \geqq 1$. On the other hand, we have $r \leqq k + n - m < n$.
Then $X_0$ is a non-trivial closed toroidal subgroup of $X$. This contradicts
our assumption.\par
Conversely, we assume that ${\mathbb T}$ is simple. Suppose that
$X$ contains a closed toroidal subgroup $X_0 = {\mathbb C}^{n_0}/\Gamma_0$
with $1 \leqq n_0 < n$. We can represent $X_0$ as a principal
$({\mathbb C}^{*})^{n_0 - m_0}$-bundle
$\rho_0 : X_0 \longrightarrow {\mathbb T}_0$ over an $m_0$-dimensional
complex torus ${\mathbb T}_0$, where ${\rm rank}\; \Gamma_0 = n_0 + m_0$
with $1 \leqq m_0 \leqq n_0$. Let $\sigma : {\mathbb C}^n \longrightarrow
{\mathbb C}_{\Gamma}^m$ and $\sigma_0 : {\mathbb C}^{n_0} \longrightarrow
{\mathbb C}_{\Gamma_0}^{m_0}$ be projections.
Then ${\mathbb T} = {\mathbb C}_{\Gamma}^m/\sigma(\Gamma)$ and
${\mathbb T}_0 = {\mathbb C}_{\Gamma_0}^{m_0}/\sigma_0(\Gamma_0)$.
Since $\Gamma_0 \subset \Gamma$, we have ${\mathbb C}_{\Gamma_0}^{m_0}
\subset {\mathbb C}_{\Gamma}^m$ and
$\sigma_0(\Gamma_0) \subset \sigma(\Gamma)\cap {\mathbb C}_{\Gamma_0}^{m_0}$.
Therefore ${\mathbb T}_0$ is a subtorus of ${\mathbb T}$. By the assumption
we have ${\mathbb T}_0 = 0$ or ${\mathbb T}$. However, $X_0$ is a toroidal group.
Then ${\mathbb T}_0 = {\mathbb T}$, hence $X_0 = X$.
This is a contradiction.
\end{proof}

Let $X$ be a quasi-abelian variety of kind 0. By Theorem 3 in \cite{abe3}
$X$ is isogenous to a product $X_1 \times \cdots \times X_k$ of
geometrically simple quasi-abelian varieties. In this case, each $X_i$ is
a quasi-abelian variety of kind 0.

\begin{theorem}
Let $X$ be a quasi-abelian variety of kind 0 which is isogenous to a product
$X_1 \times \cdots \times X_k$, where $X_i$ is a geometrically simple quasi-abelian
variety of kind 0. We assume that $X$ is isogenous to another product
$Y_1 \times \cdots \times Y_{\ell}$ of geometrically simple quasi-abelian varieties
$Y_1, \dots , Y_{\ell}$. Then, $k = \ell$ and $X_i$ is isogenous to $Y_i$ for
$i = 1, \dots ,k$ after a suitable change of indices.
\end{theorem}

\begin{proof}
Each $X_i$ is a principal $({\mathbb C}^{*})^{n_i - m_i}$-bundle over an
$m_i$-dimensional abelian variety ${\mathbb A}_i$ or ${\mathbb A}_i$ itself.
Similarly, $Y_j$ is a principal bundle over an $m'_j$-dimensional abelian
variety ${\mathbb B}_j$ or ${\mathbb B}_j$ itself. Then
$X_1 \times \cdots \times X_k$ and $Y_1 \times \cdots \times Y_{\ell}$ are
principal bundles over ${\mathbb A}_1 \times \cdots \times {\mathbb A}_k$
and ${\mathbb B}_1 \times \cdots \times {\mathbb B}_{\ell}$ respectively.
Since $X_1 \times \cdots \times X_k$ and $Y_1 \times \cdots \times Y_{\ell}$ are
isogenous, ${\mathbb A}_1 \times \cdots \times {\mathbb A}_k$
and ${\mathbb B}_1 \times \cdots \times {\mathbb B}_{\ell}$ are isogenous by
Proposition 2. By Lemma 3, ${\mathbb A}_i$ and ${\mathbb B}_j$ are simple.
Therefore, $k = \ell$ and ${\mathbb A}_i$ is isogenous to ${\mathbb B}_i$ for
$i = 1, \dots , k$ after a suitable change of indices, by a classical result for
abelian varieties. In this case,  $X_i$ and $Y_i$ are isogenous.
\end{proof}

\section{$D$-closed subfields}
Let $(z_1, \dots , z_n)$ be complex coordinates of
${\mathbb C}^n$. We consider the following system of
Briot-Bouquet type partial differential equations for
meromorphic functions $f_1, \dots ,f_n \in {\mathfrak M}(
{\mathbb C}^n)$ which are algebraically independent over
${\mathbb C}$:\par
\medskip
\noindent
For any $i,j = 1, \dots ,n$ there exists an irreducible
polynomial $P_{ij}$ such that
\begin{equation}
P_{ij}\left( \frac{\partial f_i}{\partial z_j}, f_1, \dots ,
f_n\right) = 0
\end{equation}
if $\partial f_i/\partial z_j$ is not constant.

\begin{definition}
Let $K$ be a subfield of ${\mathfrak M}({\mathbb C}^n)$
satisfying the condition {\rm (T)}. We say that $K$ is
$D$-closed if $\partial f/\partial z_i \in K$ $(i=1, \dots ,n)$
for any $f \in K$.
\end{definition}

\begin{lemma}
Let $K$ be a subfield of the rational function field
${\mathbb C}(z_1, \dots ,z_n)$ satisfying the condition {\rm (T)}.
If $K$ is $D$-closed, then $K = {\mathbb C}(z_1, \dots ,z_n)$.
\end{lemma}

\begin{proof}
For any $i = 1, \dots , n$ we set $K_i := K \cap {\mathbb C}(z_i)$.
Then $K_i$ is closed by derivative and ${\rm Trans}_{{\mathbb C}}K_i = 1$.
It suffices to show $z_i \in K_i$. Suppose that $z_i \notin K_i$.
Then $z_i$ is algebraic over $K_i$. Let
$$P(T) = T^N + A_{N-1}T^{N-1} + \cdots + A_1T + A_0 \in K_i[T]$$
be the minimal polynomial of $z_i$ over $K_i$. We have $N\geqq 2$
by the assumption. Differentiating $P(z_i)$ by $z_i$, we obtain
$$(N + A'_{N-1})z_i^{N-1} + \{ (N-1)A_{N-1} + A'_{N-2}\}z_i^{N-2} +
\cdots + (2A_2 + A'_1)z_i + A_1 + A'_0 = 0,$$
where $A'_j = dA_j/dz_i$.
Since $P(T)$ is the minimal polynomial of $z_i$, we have
$N + A'_{N-1} = 0$. Then we obtain $A_{N-1} = - Nz_i + c$ for some
constant $c \in {\mathbb C}$. Hence we have $z_i \in K_i$ for $A_{N-1} \in K_i$.
This is a contradiction.
\end{proof}

\begin{proposition}
Let $K_1$ and $K_2$ be subfields of ${\mathfrak M}({\mathbb C}^n)$
which are finitely generated over ${\mathbb C}$.
We assume that $K_1$ is non-degenerate, $K_1 \subset K_2$ and
${\rm Trans}_{{\mathbb C}}K_1 = {\rm Trans}_{{\mathbb C}}K_2$.
Then the identity mapping $id_{{\mathbb C}^n} : {\mathbb C}^n
\longrightarrow {\mathbb C}^n$ gives an isogeny
$\varphi : {\mathbb C}^n/\Gamma_{K_2} \longrightarrow 
{\mathbb C}^n/\Gamma_{K_1}$.
\end{proposition}

\begin{proof}
We have $\Gamma_{K_2} \subset \Gamma_{K_1}$ by $K_1 \subset K_2$.
Since $K_1$ is non-degenerate, $\Gamma_{K_1}$ is a discrete subgroup.
Then $\Gamma_{K_2}$ is also discrete, and $K_2$ is non-degenerate.
We set $X_1 := {\mathbb C}^n/\Gamma_{K_1}$ and $X_2 := 
{\mathbb C}^n/\Gamma_{K_2}$. A homomorphism $\varphi : X_2 \longrightarrow
X_1$ is defined from  $id_{{\mathbb C}^n}$ for $\Gamma_{K_2} \subset \Gamma_{K_1}$.\par
It suffices to show the following statement $(*)$ which means that
$\Gamma_{K_1}/\Gamma_{K_2}$ is a finite group.\par
\medskip
\noindent
$(*)$ For any $\gamma \in \Gamma_{K_1}$ there exists $\alpha \in {\mathbb N}$
with $\alpha \gamma \in \Gamma_{K_2}$.\par
\medskip
\noindent
To see this, let $t := {\rm Trans}_{{\mathbb C}}K_1 =
{\rm Trans}_{{\mathbb C}}K_2$. We take algebraically independent
functions $f_1, \dots ,f_t \in K_1$.
Since $K_2/{\mathbb C}(f_1, \dots ,f_t)$ is an algebraic extension,
there exists $g \in K_2$ such that
$K_2 = {\mathbb C}(g,f_1, \dots , f_t)$. There exists an irreducible 
polynomial $P$ over ${\mathbb C}$ such that
\begin{equation}
P(g, f_1, \dots ,f_t) = 0.
\end{equation}
Let $P_g$ and $P_{f_i}$ be the polar sets of $g$ and $f_i$ respectively.
We set
$$\Omega := {\mathbb C}^n \setminus \left(
\bigcup_{i=1}^t P_{f_i} \cup P_g \right).$$
Then $\Omega $ is an open dense subset of ${\mathbb C}^n$.
Take any $\gamma \in \Gamma_{K_1}$. It holds for all $k \in {\mathbb N}$
that
\begin{equation}
f_i(z + k\gamma) = f_i(z)\quad \text{for all $z \in {\mathbb C}^n$
and $i=1, \dots ,t$.}
\end{equation}
We fix any $z \in \Omega $. From (5.2) and (5.3) it follows that
\begin{equation}
P(g(z + k\gamma), f_1(z), \dots ,f_t(z)) = 0
\end{equation}
for all $k \in {\mathbb N}$. Since the equation
$P(T,f_1(z), \dots , f_t(z)) = 0$ has at most a finite number of solutions,
there exists a pair $(k,\ell)$ of natural numbers with $k \not= \ell$ such that
$g(z + k\gamma) = g(z + \ell \gamma)$. We note that the pair
$(k,\ell)$ depends on $z$. For any $(k,\ell) \in {\mathbb N}^2$ with $k \not= \ell$,
we set
$$A_{k,\ell} := \{ z \in \Omega ; g(z + k\gamma) = g(z + \ell \gamma)\} .$$
If the interior $A_{k,\ell}^{\circ}$ of $A_{k,\ell}$ is not empty for some
$(k,\ell)$ with $k \not= \ell$, then $\Omega = A_{k,\ell}^{\circ}$ by the
uniqueness theorem. In this case, $|k - \ell |\gamma$ is a period of $g$.
Then the statement $(*)$ holds.\par
Assume that $A_{k,\ell}^{\circ} = \emptyset$ for all $(k,\ell) \in {\mathbb N}^2$
with $k \not= \ell$. Then the set
$$A := \bigcup_{\substack{(k,\ell) \in {\mathbb N}^2\\ k \not= \ell}}
A_{k,\ell}$$
is nowhere dense in $\Omega $. However, we have $A = \Omega$ as shown above.
This contradicts Baire's category theorem. Hence we conclude the statement $(*)$.
\end{proof}

\begin{lemma}
Let $K$ be a subfield of
${\mathbb C}(e^{z_1}, \dots , e^{z_n})$ satisfying the condition {\rm (T)}.
If $K$ is $D$-closed,
then there exists a ${\mathbb C}$-linear isomorphism
$\Phi : {\mathbb C}^n \longrightarrow {\mathbb C}^n$ such that
$\Phi^{*}K = {\mathbb C}(e^{z_1}, \dots , e^{z_n})$.
\end{lemma}

\begin{proof}
We set $L := {\mathbb C}(e^{z_1}, \dots , e^{z_n})$.
Since $K \subset L$ and ${\rm Trans}_{{\mathbb C}}K = 
{\rm Trans}_{{\mathbb C}}L = n$, 
${\mathbb C}^n/\Gamma_K$ and ${\mathbb C}^n/\Gamma_L$
are isogenous by Proposition 3.
Then $\Gamma_K/\Gamma_L$ is a finite group.
Therefore, there exist $\alpha_1, \dots ,\alpha_n \in {\mathbb N}$
such that
$$\Gamma_K = (2\pi \sqrt{-1}/\alpha_1){\mathbb Z} \times
\cdots \times (2\pi \sqrt{-1}/\alpha_n){\mathbb Z} $$
for $\Gamma_L = (2\pi \sqrt{-1}{\mathbb Z})^n$.
Hence we obtain 
$K \subset {\mathbb C}(e^{\alpha_1z_1}, \dots , e^{\alpha_n z_n})$.
Once we show $K = {\mathbb C}(e^{\alpha_1z_1}, \dots , e^{\alpha_n z_n})$,
the lemma is obvious.\par
Then we show $e^{\alpha_1z_1}, \dots ,e^{\alpha_nz_n} \in K$.
For any $i = 1, \dots , n$ we set $K_i := K \cap {\mathbb C}(e^{\alpha_i z_i})$.
Then $K_i$ is closed by derivative, ${\rm Trans}_{{\mathbb C}}K_i = 1$
and $\Gamma_{K_i} = (2\pi \sqrt{-1}/\alpha_i){\mathbb Z}$. Therefore it is sufficient
to consider the case $n=1$.\par
Letting $\zeta = \alpha_i z_i$, we may consider a subfield
$K$ of ${\mathbb C}(e^{\zeta})$ which is closed by derivative,
${\rm Trans}_{{\mathbb C}}K = 1$ and $\Gamma_K = 2\pi \sqrt{-1}{\mathbb Z}$.
We show $e^{\zeta} \in K$. If $e^{\zeta} \notin K$, then $e^{\zeta}$ is algebraic
over $K$. Let $N \geqq 2$ be the degree of $e^{\zeta}$ over $K$. Take the
minimal polynomial
$$P(T) = T^N + A_{N-1}T^{N-1} + \cdots + A_1T + A_0 \in K[T]$$
of $e^{\zeta}$ over $K$. Then we have
\begin{equation}
(e^{\zeta})^N + A_{N-1}(e^{\zeta})^{N-1} + \cdots + A_1e^{\zeta} + A_0 = 0.
\end{equation}
Differentiating (5.5), we obtain
\begin{equation}
(e^{\zeta})^N + \sum _{k=1}^{N}\frac{1}{N}\left(A'_{N-k} + (N-k)A_{N-k}\right)
(e^{\zeta})^{N-k} = 0.
\end{equation}
Since the minimal polynomial is unique, we have
$$\frac{1}{N}\left( A'_{N-k} + (N-k)A_{N-k}\right) = A_{N-k}$$
for all $k = 1, \dots , N$. Then $A'_{N-k} = kA_{N-k}$. Therefore we obtain
$A_{N-k} = c_{N-k}e^{k\zeta}$ for some constant
$c_{N-k} \in {\mathbb C}$. Especially we have
$A_0 = c_0e^{N\zeta} \in K$. Since $P(T)$ is irreducible, we have $c_0 \not= 0$.
Therefore $e^{N\zeta} \in K$, hence we have
${\mathbb C}(e^{\zeta}) \supset K \supset {\mathbb C}(e^{N\zeta})$.\par
Let $M$ be the degree of the extension $K/{\mathbb C}(e^{N\zeta})$.
Then $M < N$ for $K \varsubsetneqq {\mathbb C}(e^{\zeta})$.
The period group of ${\mathbb C}(e^{N\zeta})$ is $(2\pi \sqrt{-1}/N){\mathbb Z}$.
There exists $f \in K$ such that $K = {\mathbb C}(f, e^{N\zeta})$. Let
$$Q(T) = T^M + B_{M-1}T^{M-1} + \cdots + B_1T + B_0$$
be the minimal polynomial of $f$ over ${\mathbb C}(e^{N\zeta})$, where
$B_j \in {\mathbb C}(e^{N\zeta})$ for $j = 0, 1, \dots , M-1$.
We set $\gamma_0 := 2\pi \sqrt{-1}/N$. Then $B_j(\zeta + \gamma_0) = B_j(\zeta)$
for all $\zeta \in {\mathbb C}$. Take a point $\zeta_0 \in {\mathbb C}$
at which all of $B_j(\zeta)$ are holomorphic. We take a sequence
$\{ \zeta^{(\mu)}\}$ such that all of $B_j(\zeta)$ are holomorphic at
$\zeta^{(\mu)}$ and $\zeta^{(\mu)} \rightarrow \zeta_0$ $(\mu \to \infty)$.
Fix $\zeta^{(\mu)}$. Then we have
$$f(\zeta^{(\mu)} + \alpha \gamma_0)^M + B_{M-1}(\zeta^{(\mu)})
f(\zeta^{(\mu)} + \alpha \gamma_0)^{M-1} + \cdots +
B_0(\zeta^{(\mu)}) = 0$$
for any $\alpha = 1, \dots , N$. Since $M < N$, there exists a pair
$(k,\ell)$ with $1 \leqq k < \ell \leqq N$ such that
$f(\zeta^{(\mu)} + k \gamma_0) = f(\zeta^{(\mu)} + \ell \gamma_0)$.
We note that the pair $(k,\ell)$ depends on $\zeta^{(\mu)}$. However,
we can take a subsequence $\{ \zeta^{(\mu')}\}$ of $\{ \zeta^{(\mu)}\}$
and a pair $(k,\ell)$ with $1 \leqq k < \ell \leqq N$ such that
$$f(\zeta^{(\mu')} + k\gamma_0) = f(\zeta^{(\mu')} + \ell \gamma_0)$$
for all $\mu'$. Since $\zeta^{(\mu')} \rightarrow \zeta_0$, we obtain
$$f(\zeta + k\gamma_0) = f(\zeta + \ell \gamma_0)$$
for all $\zeta \in {\mathbb C}$ by the uniqueness theorem. Then
$(\ell - k)\gamma_0$ is a period of $f$. Since $0 < \ell - k < N$,
this contradicts our assumption
$\Gamma_K = 2\pi \sqrt{-1}{\mathbb Z}$.
\end{proof}

\begin{proposition}
Let $K$ be a non-degenerate subfield of ${\mathfrak M}({\mathbb C}^n)$
satisfying the condition {\rm (T)}. 
We assume that $K$ is a $D$-closed subfield of a {\rm W}-type subfield
$\sigma^{*}({\mathfrak M}(\overline{X})|_X)$, where $X = {\mathbb C}^p \times
({\mathbb C}^{*})^q \times {\mathcal Q}$ with an $r$-dimensional quasi-abelian
variety of kind 0 and $\sigma : {\mathbb C}^n \longrightarrow X$ is the projection.
Let $z = (z',z'') = (z_1, \dots ,z_p, z_{p+1}, \dots ,z_{p+q}; z_{p+q+1}, \dots , z_n)$
be coordinates of ${\mathbb C}^n$ such that ${\mathfrak M}(\overline{X}) =
{\mathbb C}(z_1, \dots , z_p, e^{z_{p+1}}, \dots ,e^{z_{p+q}}, g_0(z''), g_1(z''),
\dots , g_r(z''))$, where $g_0(z''), g_1(z''), \dots ,$   $g_r(z'')$ are generators of
${\mathfrak M}(\overline{{\mathcal Q}})$. Then there exists a ${\mathbb C}$-linear
isomorphism $\Phi : {\mathbb C}^n \longrightarrow {\mathbb C}^n$ such that
$$\Phi ^{*}K = {\mathbb C}(z_1, \dots , z_p, e^{z_{p+1}}, \dots , e^{z_{p+q}})\cdot 
K_0,$$
where $K_0$ is a $D$-closed subfield of ${\mathfrak M}(\overline{{\mathcal Q}})|_{{\mathcal Q}}$
with ${\rm Trans}_{{\mathbb C}}K_0 = {\rm Trans}_{{\mathbb C}}
{\mathfrak M}(\overline{{\mathcal Q}})|_{{\mathcal Q}}$.
\end{proposition}

\begin{proof}
Let $K(X) := \sigma^{*}({\mathfrak M}(\overline{X})|_X)$. Then $K \subset K(X)$.
By Lemma 4 we have $K\cap {\mathbb C}(z_1, \dots , z_p) = {\mathbb C}(z_1, \dots , 
z_p)$. We set $K_1 := K \cap {\mathbb C}(e^{z_{p+1}}, \dots , e^{z_{p+q}},
g_0(z''), g_1(z''),$  $ \dots , g_r(z''))$. Then we have $K = {\mathbb C}(z_1, \dots , z_p)\cdot
K_1$. We note that $K_1$ is non-degenerate, $D$-closed and ${\rm Trans}_{{\mathbb C}}
K_1 = q + r$. \par
Since $K_1 \cap {\mathbb C}(e^{z_{p+1}}, \dots ,e^{z_{p+q}})$ is $D$-closed and
${\rm Trans}_{{\mathbb C}}(K_1 \cap {\mathbb C}(e^{z_{p+1}}, \dots ,e^{z_{p+q}})) = q$,
there exists a ${\mathbb C}$-linear isomorphism
$\Phi : {\mathbb C}^q \longrightarrow {\mathbb C}^q$ such that
$$\Phi^{*}(K_1 \cap {\mathbb C}(e^{z_{p+1}}, \dots ,e^{z_{p+q}})) =
 {\mathbb C}(e^{z_{p+1}}, \dots ,e^{z_{p+q}})) $$
by Lemma 5. If we set a ${\mathbb C}$-linear isomorphism
${\widetilde \Phi} := (id_{{\mathbb C}^p}, \Phi, id_{{\mathbb C}^r}) :
{\mathbb C}^p \times {\mathbb C}^q \times {\mathbb C}^r \longrightarrow
{\mathbb C}^p \times {\mathbb C}^q \times {\mathbb C}^r$,
then we obtain
$${\widetilde \Phi}^{*}K = {\mathbb C}(z_1, \dots , z_p, e^{z_{p+1}}, \dots ,
e^{z_{p+q}})\cdot K_0,$$
where $K_0 = K \cap {\mathbb C}(g_0(z''), g_1(z''), \dots , g_r(z''))$.
\end{proof}

Let $f_1, \dots , f_n \in {\mathfrak M}({\mathbb C}^n)$ be algebraically 
independent meromorphic functions which are solutions of a system of Briot-Bouquet type
partial differential equations (5.1). We define
$$
K := {\mathbb C}\left( f_1, \dots , f_n, \left\{ \frac{
\partial f_i}{\partial z_j}; i,j = 1, \dots , n \right\} \right).
$$
Then $K$ obviously satisfies the condition (T).
In this case, we say that $K$ is determined by solutions
$f_1, \dots , f_n$ of a system of Briot-Bouquet type
partial differential equations.
We see at once the
following lemma by (5.1).

\begin{lemma}
Let $K$ be as above. Then
$K$ is $D$-closed.
\end{lemma}

\begin{proof}
It suffices to show
$$\frac{\partial ^2f_i}{\partial z_j \partial z_k} \in K$$
for any $i,j,k = 1, \dots , n$. If $\partial f_i/\partial z_j$
is not constant, then there exists an irreducible polynomial
$P(S,T_1, \dots , T_n)$ such that
$$P\left( \frac{\partial f_i}{\partial z_j}, f_1, \dots , f_n
\right) = 0.$$
Differentiating the above equation by $z_k$, we obtain
$$
\frac{\partial ^2 f_i}{\partial z_j \partial z_k} = -
\frac{\sum _{\ell = 1}^{n} P_{T_{\ell }}\left( 
\frac{\partial f_i}{\partial z_j}, f_1, \dots , f_n \right)
\frac{\partial f_{\ell }}{\partial z_k}}
{P_S \left( \frac{\partial f_i}{\partial z_j}, f_1, \dots , f_n
\right)}
\in K.$$
\end{proof}

\section{Algebraic extension}
First we prove the following lemma.

\begin{lemma}
Let $K_0$ be a $D$-closed subfield of ${\mathfrak M}({\mathbb C}^n)$.
If $K/K_0$ is an algebraic extension, then $K$ is also a $D$-closed
subfield.
\end{lemma}

\begin{proof}
There exists $f \in K$ such that $K = K_0(f)$. Take the minimal polynomial
$P(T) = T^N + \sum _{j=1}^{N}A_{N-j}T^{N-j} \in K_0[T]$ of $f$ over $K_0$.
It suffices to show $\partial f/\partial z_k \in K$ for $k = 1, \dots , n$.
We assume that $\partial f/\partial z_k$ is not constant. Differentiating
$P(f) = 0$ by $z_k$, we obtain
$$\left( Nf^{N-1} + \sum _{j=1}^{N-1}(N -j) A_{N-j}f^{N-j-1}\right)
\frac{\partial f}{\partial z_k} + \sum _{j=1}^{N} \frac{\partial A_{N-j}}
{\partial z_k} f^{N-j} = 0.$$
Therefore we have
$$\frac{\partial f}{\partial z_k} = - 
\frac{\sum_{j=1}^{N}\frac{\partial A_{N-j}}{\partial z_k} f^{N-j}}
{Nf^{N-1} + \sum_{j=1}^{N-1}(N - j) A_{N-j} f^{N - j-1}}.$$
Since $K_0$ is $D$-closed, $\partial f/\partial z_k$ belongs to $K$.
\end{proof}

We recall a classical fact on meromorphic functions to show the next
proposition. The following lemma is an immediate consequence from
the continuation theorem of Levi for meromorphic functions (cf. \cite{levi},
\cite{kneser} and \cite{okuda-sakai}).

\begin{lemma}
Let $f(z,w)$ be a meromorphic function on
$\{ z \in {\mathbb C}^N ; \| z \| < R\} \times \{ w \in {\mathbb C}
; 0 < |w| < r\}$. We assume that $(0,0)$ is an essential singularity of $f$.
Then, for any $z^0 \in {\mathbb C}^N$ close to $0$, $f$ has an 
essential singularity in
$\{ z^0\} \times \{ w \in {\mathbb C} ; |w| < r\}$.
\end{lemma}

\begin{proposition}
Let $K$ be a non-degenerate subfield of a {\rm W}-type subfield of 
${\mathfrak M}({\mathbb C}^n)$ with ${\rm Trans}_{{\mathbb C}}K = n$.
If $\widetilde{K}/K$ is an algebraic extension, then
there exists a ${\mathbb C}$-linear isomorphism $\Phi : {\mathbb C}^n
\longrightarrow {\mathbb C}^n$ such that $\Phi^{*}\widetilde{K}$ is
a subfield of a {\rm W}-type subfield of ${\mathfrak M}({\mathbb C}^n)$.
\end{proposition}

\begin{proof}
Let $X$ and $K(X) = \sigma^{*}({\mathfrak M}(\overline{X})|_X)$
be the same as in Proposition 4.
We write $X = {\mathbb C}^n/\Gamma = {\mathbb C}^p \times
({\mathbb C}^{*})^q \times {\mathcal Q}$, where ${\mathcal Q} =
{\mathbb C}^r/\Gamma_0$.
We note $\Gamma \subset \Gamma_K$ for $K \subset K(X)$. Since
${\rm Trans}_{{\mathbb C}}\widetilde{K} = {\rm Trans}_{{\mathbb C}}K = n$,
there exists a ${\mathbb C}$-linear isomorphism $\Phi : {\mathbb C}^n
\longrightarrow {\mathbb C}^n$ such that
$\Phi^{*}\widetilde{K} \subset \sigma_K^{*}{\mathfrak M}({\mathbb C}^n/
\Gamma_K)$ by Proposition 3, where $\sigma_K : {\mathbb C}^n \longrightarrow
{\mathbb C}^n/\Gamma_K$ is the projection. We have
$\sigma_K^{*}{\mathfrak M}({\mathbb C}^n/\Gamma_K) \subset
\sigma^{*}{\mathfrak M}(X)$ for $\Gamma \subset \Gamma_K$.
Then we obtain $\Phi^{*}\widetilde{K} \subset \sigma^{*}{\mathfrak M}(X)$.\par
Therefore, it is sufficient to show that
$\Phi^{*}\widetilde{K} \subset K(X) = \sigma^{*}({\mathfrak M}(\overline{X})|_X)$.
We have $\Phi^{*}K \subset K(X)$ by $K \subset \widetilde{K}$,
$K \subset K(X)$ and Proposition 2.
Suppose that there exists $g \in \Phi^{*}\widetilde{K}$ such that $g$ is not
meromorphically extendable to $\overline{X}$.
Let $f_1, \dots , f_n \in \Phi^{*}K$ be algebraically independent. 
Since $\Phi^{*}\widetilde{K}/\Phi^{*}K$ is an algebraic extension,
there exists an irreducible polynomial
$$P(S_1, \dots , S_n,T) = \sum _{j=0}^{N}
A_j(S_1, \dots , S_n)T^j$$
with $N \geqq 1$ such that
\begin{equation}
P(f_1, \dots , f_n, g) = 0,
\end{equation}
where $A_j(S_1, \dots , S_n) \in {\mathbb C}[S_1, \dots , S_n]$.
Since $g$ is not meromorphically extendable to $\overline{X}$,
it has an essential singularity $a \in \overline{X} \setminus X$.
\par
Let ${\rm rank}\; \Gamma _0 = r + s$ $(1 \leqq s \leqq r)$. Then
$\overline{X}$ is a fiber bundle $\overline{\rho } : \overline{X}
\longrightarrow {\mathbb A}$ over an $s$-dimensional abelian
variety ${\mathbb A}$ with fibers $({\mathbb P}^1)^{p + q + r -s}$. 
We set $x_0 := \overline{\rho }(a) \in {\mathbb A}$. We take coordinates
$(w_1, \dots ,w_p,w_{p+1}, \dots , w_{p+q}, w_{p+q+1}, \dots ,
w_{n-s})$ on the fiber $\overline{\rho }^{-1}(x_0)$. Then
$a \in \overline{\rho }^{-1}(x_0)$ is represented as
$$a = (a_1, \dots , a_{k-1},\varepsilon , a_{k+1}, \dots , a_{n-s})$$
in these coordinates, where $\varepsilon = 0$ or $\infty $ and
$1 \leqq k \leqq n-s$.
We set
$$a' := (a_1, \dots , a_{k-1},a_{k+1}, \dots , a_{n-s}).$$
We may assume $a' \in {\mathbb C}^{n-s-1}$ by Lemma 8. Put
$$L_{a'} := \{ w_1 = a_1, \dots , w_{k-1} = a_{k-1},
w_{k+1} = a_{k+1}, \dots , w_{n-s} = a_{n-s}\}.$$
Then $L_{a'}$ is a complex line in $\overline{\rho }^{-1}(x_0)$
and $g|_{L_{a'}}(w_k)$ has an essential singularity at
$w_k = \varepsilon $. 
Taking another $a$ if necessary, we may further
assume that $L_{a'}$ is not contained in both the zero set of
$A_j(f_1, \dots ,f_n)$ and the polar set of $A_j(f_1, \dots , f_n)$
for all $j=0, 1, \dots , N$, by Lemma 8. In this case, $w_k = \varepsilon $ is
a holomorphic point or a pole of $A_j(f_1, \dots , f_n)|_{L_{a'}}$.
We have the following possibilities.
\par
\medskip
\noindent
(i) $A_j(f_1, \dots , f_n)|_{L_{a'}}$ is holomorphic at
$w_k = \varepsilon $ for all $j=0, 1, \dots , N$.\\
(ii) There exists $i$ such that $A_i(f_1, \dots , f_n)|_{L_{a'}}$
has a pole at $w_k = \varepsilon $.\par
\medskip
\noindent
We may assume $A_N(f_1, \dots , f_n)|_{L_{a'}}(\varepsilon ) \not= 0$
in the case (i) using Lemma 8 again if necessary. Let $i_1, \dots , i_{\ell }$ with
$0 \leqq i_1 < \cdots < i_{\ell } \leqq N$ be all $i$ possessing the
property in the case (ii).
We denote by $k_{\mu }$ the order of pole of 
$A_{i_{\mu }}(f_1, \dots , f_n)|_{L_{a'}}$ at $w_k = \varepsilon $
for all $\mu = 1, \dots , \ell $. We set
$$
m:=
\begin{cases}
0 & \text{in the case (i)},\\
\max \{ k_1, \dots , k_{\ell }\} &
\text{in the case (ii)}\\
\end{cases}
$$
and
$$
N_0 :=
\begin{cases}
N & \text{in the case (i)},\\
\max \{ i_{\mu } ; m = k_{\mu }\} &
\text{in the case (ii)}.
\end{cases}
$$
By the definition of $m$, $w_k^{\pm m} A_j(f_1, \dots , f_n)|_{
L_{a'}} (w_k)$ is holomorphic at $w_k = \varepsilon $ for all
$j=0,1, \dots ,N$, where $\pm $ is determined according to
$\varepsilon = 0$ or $\infty $. Then, we have
\begin{equation}
\sum _{j=0}^{N} w_k^{\pm m} A_j(f_1, \dots , f_n)|_{L_{a'}}(w_k)
T^j \longrightarrow P_0(T)
\end{equation}
as $w_k \to \varepsilon $, where $P_0(T)$ is a polynomial of 
degree $N_0$. Take $c \in {\mathbb C}$ such that $c$ is neither
a solution of $P_0(T)=0$ nor an exceptional value of $g|_{L_{a'}}$
at $w_k = \varepsilon $. Then there exists a sequence
$\{ w_k^{(\nu )} \} \subset L_{a'}$ with $w_k^{(\nu )} \to a$ such
that $g(w_k^{(\nu )}) = c$ by Picard's big theorem. It follows from
(6.1) that
\begin{equation*}
\begin{split}
0 & = P(f_1(w_k^{(\nu )}), \dots , f_n(w_k^{(\nu )}), g(w_k^{(\nu )}))\\
& = P(f_1(w_k^{(\nu )}), \dots , f_n(w_k^{(\nu )}), c).
\end{split}
\end{equation*}
Letting $\nu \to \infty $, we obtain $P_0(c)=0$ by (6.2). This
contradicts the choice of $c$. Therefore, any $g \in \Phi^{*}\widetilde{K}$
is meromorphically extendable to $\overline{X}$.
\end{proof}

We are in a position to prove Theorem 1.
\par
\medskip
\noindent
{\it Proof of Theorem 1.}
The implications (3) $\Rightarrow $ (2) $\Rightarrow $ (1)
are trivial.\par
Assume that (1) holds. Let $f$ be a non-constant function in $K$
with $f' \in K$. Put $K_0 := {\mathbb C}(f,f')$. Then, by the Briot-Bouquet
theorem $K_0$ is a subfield of ${\mathbb C}(\zeta)$, ${\mathbb C}(e^{\alpha \zeta})$
$(\alpha \in {\mathbb C}^{*})$ or an elliptic function field.
We may assume $\alpha = 1$ by a linear change of variable.
By Lemma 6 $K_0$ is closed by derivative. Since $K/K_0$ is an
algebraic extension, $K$ is also closed by derivative by Lemma 7.
It follows from Proposition 5 that there exists a ${\mathbb C}$-linear isomorphism
$\Phi : {\mathbb C} \longrightarrow {\mathbb C}$ such that
$\Phi^{*}K$ is a subfield of ${\mathbb C}(\zeta)$, ${\mathbb C}(e^{\zeta})$ or
an elliptic function field, which is closed by derivative.
Using Lemmas 4 and 5, we obtain that $\Phi^{*}K$ is ${\mathbb C}(\zeta)$,
${\mathbb C}(e^{\zeta})$ or a subfield of an elliptic function field which is
closed by derivative. Thus the implication (1) $\Rightarrow $ (3)
is proved.
\hfill
$\Box$

\section{Main theorem}
In the previous sections we assumed that a subfield $K$ of
${\mathfrak M}({\mathbb C}^n)$ is non-degenerate. The following
example shows that this condition is necessary in our argument.

\begin{example}
We define $f(z) := z_1$ and $g(z) := e^{z_1^2}$ for
$z = (z_1, z_2) \in {\mathbb C}^2$. Put $K := {\mathbb C}(f,g)$.
Then $K$ satisfies the condition {\rm (T)} as a subfield of
${\mathfrak M}({\mathbb C}^2)$. It is obviously $D$-closed.
However, it does not become a subfield of a {\rm W}-type subfield
by any linear change of variables.
\end{example}

Furthermore, the next example shows that another condition is
needed.

\begin{example}
Let $f(z) := z_1$ and $g(z) := e^{z_1^2}e^{z_2}$ for
$z = (z_1,z_2) \in {\mathbb C}^2$. If we set $K := {\mathbb C}(f,g)$,
then $K$ satisfies the condition {\rm (T)}. It is non-degenerate
and $D$-closed, but does not become a subfield of a {\rm W}-type subfield
by any linear change of variables.
\end{example}

Let $K$ be a non-degenerate subfield of ${\mathfrak M}({\mathbb C}^n)$ with
the period group $\Gamma _K$. Then 
$K \subset \sigma_K^{*}{\mathfrak M}({\mathbb C}^n/\Gamma_K)$, where
$\sigma _{K} : {\mathbb C}^n \longrightarrow
{\mathbb C}^n/\Gamma _{K}$ is the projection.
By the Remmert-Morimoto theorem we have
\begin{equation*}
{\mathbb C}^n/\Gamma _{K} = {\mathbb C}^p \times ({\mathbb C}^{*})^q
\times ({\mathbb C}^r/\Gamma _0),
\end{equation*}
where ${\mathbb C}^r/\Gamma _0$ is a toroidal group with
${\rm rank}\; \Gamma _0 = r+s$ and $n=p+q+r$.
We note that ${\mathbb C}^r/\Gamma_0$ is a quasi-abelian variety,
because there exists a non-degenerate meromorphic function on
${\mathbb C}^r/\Gamma_0$ (for example, see Theorem 5.1.10 in
\cite{abe4}).
From the structure of fiber bundle
$\rho _0 : {\mathbb C}^r/\Gamma _0 \longrightarrow {\mathbb T}$
over an $s$-dimensional complex torus ${\mathbb T}$ with fibers
$({\mathbb C}^{*})^{r-s}$, we obtain a fiber bundle
$\rho _{K} : {\mathbb C}^n/\Gamma _{K} \longrightarrow {\mathbb T}$
with fibers ${\mathbb C}^p \times ({\mathbb C}^{*})^q \times
({\mathbb C}^{*})^{r-s}$.
The next required condition is the following condition {\rm (D)}
concerned with degeneration of the transcendence degree.\par
\medskip
\noindent
{\rm (D)} If $L$ is a complex line in ${\mathbb C}^n$ such that
$\sigma _{K}(L)$ is a factor of the fiber
$\rho _{K}^{-1}(t)$ for some $t \in {\mathbb T}$, then
${\rm Trans}_{{\mathbb C}}K|_L = 1$.
\par

\begin{proposition}
If $K$ is a non-degenerate subfield of a ${\rm W}$-type subfield
$K_W$ with ${\rm Trans}_{{\mathbb C}}K = {\rm Trans}_{{\mathbb C}}K_W$, then it satisfies
the condition ${\rm (D)}$.
\end{proposition}

\begin{proof}
We first show that $K_W$ satisfies the condition ${\rm (D)}$.
Let ${\rm Trans}_{{\mathbb C}}K_W = n$.
By Definition 5 we may write
$K_W = \sigma ^{*}({\mathfrak M}(\overline{X})|_X)$, where
$X = {\mathbb C}^n/\Gamma_{K_W}  = {\mathbb C}^p \times ({\mathbb C}^{*})^q \times {\mathcal Q}$
and $\sigma : {\mathbb C}^n \longrightarrow X$ is the projection.
An $r$-dimensional quasi-abelian variety
${\mathcal Q} = {\mathbb C}^r/\Gamma _0$ of kind 0 is a fiber
bundle $\rho _0 : {\mathcal Q} \longrightarrow {\mathbb A}$ over
an $s$-dimensional abelian variety ${\mathbb A}$, where
${\rm rank}\; \Gamma _0 = r+s$. Then we have a fiber bundle
$\rho : X \longrightarrow {\mathbb A}$ with fibers
${\mathbb C}^p \times ({\mathbb C}^{*})^q \times 
({\mathbb C}^{*})^{r-s}$. Let $L$ be a complex line in
${\mathbb C}^n$ such that $\sigma (L)$ is a factor of the fiber
$\rho ^{-1}(a)$ for some $a \in {\mathbb A}$. Then, it is obvious
that ${\rm Trans}_{{\mathbb C}}K_W|_L = 1$ by the definition of
${\mathfrak M}(\overline{X})|_X$.\par
Next we consider the general case. Let $K$ be a subfield of $K_W$ with
${\rm Trans}_{{\mathbb C}}K = {\rm Trans}_{{\mathbb C}}K_W = n$.
By Proposition 3 the identity mapping $id_{{\mathbb C}^n}$ gives
an isogeny $\varphi : X \longrightarrow
X_K := {\mathbb C}^n/\Gamma_K$.
Since $X$ and $X_K$ are isogenous, we can write 
$X_K = {\mathbb C}^p \times ({\mathbb C}^{*})^q \times {\mathcal Q}_K$,
where ${\mathcal Q}_K = {\mathbb C}^r/(\Gamma_K)_0$ is a quasi-abelian
variety of kind 0. Then $X_K$ is a fiber bundle 
$\rho_K : X_K \longrightarrow {\mathbb A}_K$ over an $s$-dimensional
abelian variety ${\mathbb A}_K$ with fibers
${\mathbb C}^p \times ({\mathbb C}^{*})^q \times ({\mathbb C}^{*})^{r-s}$.
Let $L$ be a complex line in ${\mathbb C}^n$ such that $\sigma_K(L)$
is a factor of the fiber $\rho_K^{-1}(a)$ for some $a \in {\mathbb A}_K$.
By Proposition 2, $id_{{\mathbb C}^n}$ also gives an isogeny
$\varphi_{{\mathbb A}} : {\mathbb A} \longrightarrow {\mathbb A}_K$.
Then $L$ is considered as a factor of $\rho^{-1}(\widetilde{a})$ 
for some $\widetilde{a}$  with $\varphi_{{\mathbb A}}(\widetilde{a}) = a$.
Since $K_W$ satisfies the
condition {\rm (D)}, we have
${\rm Trans}_{{\mathbb C}}K_W|_{L} = 1$.
Then we obtain ${\rm Trans}_{{\mathbb C}}K|_{L} = 1$.
This finishes the proof.
\end{proof}

The following theorem is a generalization of Theorem 1. When
$n=1$, we cannot recognize the condition {\rm (D)},
because it is meaningless in this case.

\begin{theorem}
Let $K$ be a non-degenerate subfield of ${\mathfrak M}({\mathbb C}^n)$
satisfying the condition {\rm (T)}.
Then the following statements are equivalent.\\
(1) $K$ satisfies the condition {\rm (D)} and is an algebraic extension
of a field determined by solutions $f_1, \dots , f_n$
of a system of Briot-Bouquet type partial
differential equations.\\
(2) $K$ satisfies the condition {\rm (D)} and is $D$-closed.\\
(3) There exist a {\rm W}-type subfield
$\sigma^{*}({\mathfrak M}(\overline{X})|_X)$ with
$X = {\mathbb C}^p \times ({\mathbb C}^{*})^q \times {\mathcal Q}$ and
an isogeny $\varphi : X \longrightarrow {\mathbb C}^n/\Gamma_K$ such that
$$\Phi^{*}K = {\mathbb C}(z_1, \dots , z_p, e^{z_{p+1}}, \dots , e^{z_{p+q}})
\cdot K_0,$$
where $\Phi$ is the linear extension of $\varphi$ and $K_0$ is a $D$-closed
subfield of $\sigma_{{\mathcal Q}}^{*}({\mathfrak M}(\overline{{\mathcal Q}})|_{{\mathcal Q}})$
with the projection $\sigma_{{\mathcal Q}} : {\mathbb C}^r \longrightarrow {\mathcal Q}$.
\end{theorem}

\begin{proof}
The implication (1) $\Rightarrow$ (2) is an immediate consequence from
Lemmas 6 and 7.\par
We assume (2). We set
$$X : = {\mathbb C}^n/\Gamma_K = {\mathbb C}^p \times ({\mathbb C}^{*})^q
\times ({\mathbb C}^r/\Gamma_0).$$
Since there exists a non-degenerate meromorphic function on $X$,
${\mathbb C}^r/\Gamma_0$ is a quasi-abelian variety.
Let ${\rm rank}\; \Gamma_0 = r + s$. The quasi-abelian variety
${\mathbb C}^r/\Gamma_0$ has the structure of the standard principal
$({\mathbb C}^{*})^{r-s}$-bundle
$\rho_0 : {\mathbb C}^r/\Gamma_0 \longrightarrow {\mathbb T}$ over
an $s$-dimensional complex torus ${\mathbb T}$ as a toroidal group. Therefore, $X$
is a fiber bundle $\rho : X \longrightarrow {\mathbb T}$ with fibers
${\mathbb C}^p \times ({\mathbb C}^{*})^q \times ({\mathbb C}^{*})^{r-s}$.
The standard compactification $\overline{X}$ of $X$ is the associated
$({\mathbb P}^{1})^{p+q+r-s}$-bundle $\overline{\rho} : \overline{X} \longrightarrow
{\mathbb T}$. Let $\sigma : {\mathbb C}^n \longrightarrow X$ be
the projection. Then there exists a subfield $\kappa$ of ${\mathfrak M}(X)$
such that $K = \sigma^{*}\kappa$.\par
We show $\kappa \subset {\mathfrak M}(\overline{X})|_X$. It suffices
to show that every $f \in \kappa $ is meromorphically
extended to a compactification $({\mathbb P}^1)^{p+q+r-s}$ of
$\rho ^{-1}(t) = {\mathbb C}^p \times ({\mathbb C}^{*})^{q+r-s}$
for any $t \in {\mathbb T}$. Let $C$ be a factor of $\rho ^{-1}(t)$.
We set $L := \sigma ^{-1}(C)$. Then
${\rm Trans}_{{\mathbb C}}K|_L = 1$ by the condition {\rm (D)}.
Furthermore, $K|_L$ is closed by derivative.
Then $f |_{C}$ is meromorphically extended to ${\mathbb P}^1$ by Theorem 1.
Therefore, $f|_{\rho ^{-1}(t)}$ is meromorphically extended to
$({\mathbb P}^1)^{p+q+r-s}$ (Proposition 6.4 in \cite{abe1}, see also
Proposition 6.6.4 in \cite{abe4}).
Hence we have
$\kappa \subset {\mathfrak M}(\overline{X})|_X$.\par
We generally have ${\rm Trans}_{{\mathbb C}}{\mathfrak M}(\overline{X})
\leqq \dim \overline{X} = n$.
On the other hand, we have ${\rm Trans}_{{\mathbb C}}K = n$.
Then we obtain ${\rm Trans}_{{\mathbb C}}K = {\rm Trans}_{{\mathbb C}}
{\mathfrak M}(\overline{X}) = n$. Therefore, $\overline{X}$ is
projective algebraic. Hence ${\mathbb C}^r/\Gamma _0$ is a
quasi-abelian variety of kind 0 (\cite{abe2}). 
Then $K$ is a $D$-closed subfield of a {\rm W}-type subfield.
By Proposition 4 the statement (3) holds.\par
Lastly we show the implication (3) $\Rightarrow$ (1). We assume the
statement (3). By Proposition 6 $\Phi^{*}K$ satisfies the condition {\rm (D)}.
Since $\Phi$ is the linear extension of an isogeny $\varphi$ and
isogenies are fiber preserving (Proposition 2), $K$ also satisfies the
condition {\rm (D)}. We have
algebraically independent functions
$f_1, \dots , f_n \in K$. We set
\begin{equation*}
K_0 = {\mathbb C} \left( f_1, \dots , f_n, \left\{
\frac{\partial f_i}{\partial z_j } ; i, j = 1, \dots , n\right\} \right).
\end{equation*}
We note that $f_1, \dots , f_n$ satisfy a system of Briot-Bouquet type
partial differential equations.
Then $K_0$ is a $D$-closed subfield of $K$ by Lemma 6. 
It is obvious that $K/K_0$ is an algebraic extension.
\end{proof}

As Theorem 2, the above theorem shows that a {\rm W}-type subfield
is essentially the maximal subfield which satisfies the conditions {\rm (T)}
and {\rm (D)} and is $D$-closed.\par
The relation between $D$-closed subfields and subfields admitting {\rm (AAT)}
is the following.

\begin{proposition}
Let $K$ be a non-degenerate subfield of ${\mathfrak M}({\mathbb C}^n)$
satisfying the condition {\rm (T)}. If $K$ admits {\rm (AAT)}, then
the statement (2) in Theorem 4 holds for
$K$.
\end{proposition}

\begin{proof}
The condition {\rm (D)} is nothing but Proposition 5.1 in \cite{abe1}.\par
We next show that $K$ is $D$-closed.
Let $K = {\mathbb C}(f_0, f_1, \dots ,f_n)$. For any $g \in K$ and any fixed
$a \in {\mathbb C}^n$, we have $g(z+a) \in K$ by {\rm (AAT)}. Then
$$K = {\mathbb C}(f_0(z+a), f_1(z+a), \dots ,f_n(z+a))$$
for any fixed $a \in {\mathbb C}^n$. Therefore, we may assume that
$f_0(z), f_1(z), \dots , f_n(z)$ are holomorphic at $z = 0$. For any
$i = 0, 1, \dots ,n$ there exists a rational function $R \in
{\mathbb C}(S_0, S_1, \dots , S_n, T_0, T_1, \dots , T_n)$ such that
\begin{equation}
f_i(z + w) = R(f_0(z), f_1(z), \dots , f_n(z), f_0(w), f_1(w), \dots , f_n(w))
\end{equation}
for all $z,w \in {\mathbb C}^n$. Differentiating both sides of (7.1) by
$w_k$, we obtain
$$\frac{\partial f_i}{\partial z_k}(z + w) = \sum _{j=0}^{n}
\frac{\partial R}{\partial T_j}(f_0(z), f_1(z), \dots , f_n(z), f_0(w), f_1(w),
\dots ,f_n(w)) \frac{\partial f_j}{\partial w_k}(w).$$
If we set $w = 0$, then we have $\partial f_i/\partial z_k \in K$.
\end{proof}





\flushleft{
Graduate School of Science and Engineering for Research\\
University of Toyama\\
Toyama 930-8555, Japan\\}

\noindent
e-mail: abe@sci.u-toyama.ac.jp\\

\end{document}